\documentclass[12pt,a4paper]{amsart}
\usepackage{url,amssymb,latexsym}
\usepackage{fullpage} 
\newtheoremstyle{plain} 
     {2ex}              
     {2ex}              
     {}         
     {}                 
     {\bfseries}        
     {}                 
     {1ex}              
     {\thmname{#1 }\thmnumber{#2}\thmnote{ \normalfont{(#3)}}.
}
\newtheoremstyle{remark}
     {2ex}              
     {2ex}              
     {}                 
     {}                 
     {\bfseries}        
     {}                 
     {1ex}              
     {\thmname{#1 }\thmnumber{#2}\thmnote{ \normalfont{(#3)}}.
}
\theoremstyle{plain}
\newtheorem{definition}{Definition}[section]

\newtheorem{remark}[definition]{Remark}
\newtheorem{proposition}[definition]{Proposition}
\newtheorem{lemma}[definition]{Lemma}
\newtheorem{theorem}[definition]{Theorem}

\newtheorem{corollary}[definition]{Corollary}

\newtheorem{example}[definition]{Example}

\numberwithin{equation}{section}
\newcommand{\nats}{\mathbb{N}}
\newcommand{\pnats}{\ensuremath{\nats_+}}
\newcommand{\ints}{\mathbb{Z}}

\newcommand{\reals}{\mathbb{R}}
\newcommand{\preals}{\ensuremath{\reals_+}}

\newcommand{\comps}{\mathbb{C}}

\newcommand{\rproj}{\ensuremath{\mathbb{RP}}}
\newcommand{\cproj}{\ensuremath{\mathbb{CP}}}

\newcommand{\del}{\partial}

\newcommand{\me}{\mathrm{e}}
\newcommand{\mi}{\mathrm{i}}
\DeclareMathOperator{\sgn}{sgn}

\newcommand{\hl}[1]{\textnormal{\textbf{#1}}} 
\newcommand{\alphabar}{\ensuremath{\overline{\alpha}}}
\newcommand{\alphahat}{\ensuremath{\widehat{\alpha}}}
\newcommand{\chat}{\ensuremath{\widehat{c}}}
\newcommand{\cl}[1]{\ensuremath{\overline{#1}}}
\newcommand{\domain}{\ensuremath{\mathcal{D}}}
\newcommand{\gtilde}{\ensuremath{\widetilde{g}}}
\newcommand{\gtildeunder}{\ensuremath{\underline{\gtilde}}}
\newcommand{\htilde}{\ensuremath{\widetilde{h}}}
\newcommand{\sing}{\ensuremath{\mathcal{V}}}
\newcommand{\supp}{\operatorname{supp}}
\newcommand{\utilde}{\ensuremath{\widetilde{u}}}

\newcommand{\xhat}{\ensuremath{\widehat{x}}}
\newcommand{\Xtilde}{\ensuremath{\widetilde{X}}}
\begin{document}
\title[Improvements for smooth points]
  {Asymptotics of coefficients of multivariate generating functions: \\
   improvements for smooth points}
\author{Alexander Raichev}
\address{Department of Computer Science \\ University of Auckland \\
  Private Bag 92019 \\ Auckland \\ New Zealand}
\email{raichev@cs.auckland.ac.nz}
\author{Mark C. Wilson}
\address{Department of Computer Science \\ University of Auckland \\
  Private Bag 92019 \\ Auckland \\ New Zealand}
\email{mcw@cs.auckland.ac.nz}
\begin{abstract}
Let $\sum_{\beta\in\nats^d} F_\beta x^\beta$ be a multivariate power series.
For example $\sum F_\beta x^\beta$ could be a generating 
function for a combinatorial class.
Assume that in a neighbourhood of the origin this series represents a nonentire
function $F=G/H^p$ where $G$ and $H$ are holomorphic and $p$ is a positive
integer.
Given a direction $\alpha\in\pnats^d$ for which the asymptotics are
controlled by a smooth point of the singular variety $H = 0$,
we compute the asymptotics of $F_{n \alpha}$ as $n\to\infty$.
We do this via multivariate singularity analysis and give an explicit
uniform formula for the full asymptotic expansion.
This improves on earlier work of R. Pemantle and the second author
and allows for more accurate numerical approximation, as demonstrated by our
our examples (on lattice paths, quantum random walks, and nonoverlapping 
patterns).
\end{abstract}
\date{14 July 2008}
\maketitle
\section{Introduction}\label{sec:intro}

Let $\sum_{\beta\in\nats^d} F_\beta x^\beta$ be a multivariate power series.
For example $\sum F_\beta x^\beta$ could be a generating 
function for a combinatorial class.
In \cite{PeWi2002, PeWi2004} Pemantle and Wilson derived asymptotic expansions 
for the coefficients $F_\beta$ as $\beta\to\infty$ for large classes of 
series that arise often in applications.
In this article we further their program of asymptotics of coefficients of 
multivariate generating functions.

Assume that in a neighbourhood of the origin the power series 
$\sum_{\beta\in\nats^d} F_\beta x^\beta$ is the Maclaurin series of a nonentire
function $F=G/H^p$ where $G$ and $H$ are holomorphic and $p$ is a positive
integer.
For example $F$ could be a rational function.
Using multivariate singularity analysis we derive the asymptotics of 
$F_{n \alpha}$ for $\alpha\in\pnats^d$ and $n\to\infty$ in the case that these
asymptotics are controlled by smooth points of the singular variety $H = 0$.

Our presentation is organized as follows.
In Section~\ref{sec:notation} we set our notation and basic definitions.
In Section~\ref{sec:asymptotics} we give two explicit formulas for all the terms 
in the asymptotic expansion of $F_{n \alpha}$.
This improves upon \cite{PeWi2002} which gave formulas for the leading term only
and only for the case $p=1$.
Furthermore, we prove that the expansions for $F_{n\alpha}$ are uniform in 
$\alpha$ as claimed in \cite{PeWi2002}.
In Section~\ref{sec:original} we go on to express our formulas in terms of the 
original data $G$ and $H$ for easier use in calculation. 
In Section~\ref{sec:examples} we apply our formulas to several representative 
combinatorial examples (lattice paths, quantum random walks, and nonoverlapping
patterns) and demonstrate the greater numerical accuracy of 
many-term asymptotic expansions over single-term expansions. 
Maple 11 worksheets are provided.
Lastly, in Section~\ref{sec:analysis} we present the two theorems from analysis 
on Fourier-Laplace integrals that we use in our proofs.
\section{Notation and definitions}\label{sec:notation}

Let $\pnats$ and $\preals$ denote the set of positive natural numbers and 
positive real numbers, respectively.
For $r\in\reals$ and $k\in\nats = \pnats\cup\{ 0 \}$ set
$r^{\overline{k}} = r(r+1)\cdots(r+k-1)$, the $k$th rising factorial power of 
$r$, with the convention that $r^{\overline{0}}=1$.
For $m\in\pnats$, $x\in\comps^m$, and $i\le m$ let $x_i$ denote component $i$
of $x$ and $\xhat = (x_1,\ldots,x_{m-1})$.
For $\alpha\in\pnats^m$, $x\in\comps^m$, and $n\in\pnats$ define 
$\alpha+1 = (\alpha_1+1,\ldots,\alpha_m+1)$,
$\alpha!  = \alpha_1!\cdots\alpha_m!$,
$n\alpha = (n\alpha_1,\ldots,n\alpha_m)$,
$x^\alpha = x_1^{\alpha_1}\cdots x_m^{\alpha_m}$,
and $\del^\alpha = \del_1^{\alpha_1} \cdots \del_m^{\alpha_m}$,
where $\del_j$ is partial differentiation with respect to component $j$.
For $c\in\comps^m$ let 
$D(c) =\{x\in\comps^m : \forall j\le m\; |x_j|<|c_j|\}$ and
$C(c) =\{x\in\comps^m : \forall j\le m\; |x_j|=|c_j|\}$,
the polydisc and polycircle centred at the origin with polyradius 
$(|c_1|,\ldots,|c_m|)$, respectively.
Finally, for $z\in\comps\setminus\{0\}$ with $\arg z \in [-\pi/2,\pi/2]$ we set
$z^{-1/v} = |z|^{-1/v}\me^{-\mi\arg z/v}$.

Fix $d\in\pnats$ and let $G,H:\domain\to\comps$ be holomorphic functions
on a nonempty domain $\domain\subseteq\comps^d$ (an open connected set).
Assume that $G$ and $H$ are relatively prime in the ring of holomorphic
functions on $\domain$.
Let $p\in\pnats$ and define $F=G/H^p$. 
Then $F$ is holomorphic on $\domain\setminus\sing$, where $\sing$ is
the analytic variety $\{x\in\domain : H(x)=0\}$.
By \cite[Example 4.1.5,Corollary 4.2.2]{Sche2005} 
the variety $\sing$ has dimension $d-1$ and $\domain\setminus\sing$ is a domain.
Assume $0\in\domain\setminus\sing$ and $\sing\neq\emptyset$.
Let $\sum_{\beta\in\nats^d} F_\beta x^\beta$ be the Maclaurin series of $F$,
so $\del^\beta F(0)/\beta! = F_\beta$ for all $\beta\in\nats^d$.

We will derive asymptotics for $F_\beta$ as $\beta\to\infty$ along straight
lines through the origin and off the axes, that is, for $F_{n\alpha}$ with
$\alpha\in\pnats^d$ and $n\to\infty$.
For $d=2$ and $\beta\to\infty$ along more general paths see \cite{Llad2006}.

First we recall some key definitions from \cite{PeWi2002}.
Just as in the univariate case, the asymptotics for the Maclaurin coefficients 
of $F$ are determined by the location and type of singularities of $F$, that is,
by the geometry of $\sing$.
Generally the singularities closest to the origin are the most important.
We define `closest to the origin' in terms of polydiscs.

\begin{definition}
Let $c\in\sing$.
We say that $c$ is \hl{minimal} if there is no point $x\in\sing$ such that 
$|x_j|<|c_j|$ for all $j\le d$.
We say that $c$ is \hl{strictly minimal} if there is a unique $x\in\sing$ such 
that $|x_j| \leq |c_j|$ for all $j$, namely $x = c$, and we say
that $c$ is \hl{finitely minimal} if there are finitely many such values of $x$.
\end{definition}

The variety $\sing$ always contains minimal points.
To see this let $c\in\sing$ and define $f:\sing\cap\cl{D(c)}\to\reals$ by
$f(x)=\sqrt{x_1^2 +\cdots+ x_d^2}$.
Since $f$ is a continuous function on a compact space, it has a minimum, 
and that minimum is a minimal point of $\sing$.

We focus on the singularities of $F$ with the simplest geometry, namely the
regular or smooth points of $\sing$.
For a summary of what is known about non-smooth points see the survey 
\cite{PeWi2008}.

\begin{definition}
A point $c\in\sing$ is called \hl{smooth} if 
$(\del_1 H(x),\ldots,\del_d H(x)) \neq 0$ for all $x$ in a neighborhood of $c$.
\end{definition}

Equivalently, a point $c\in\sing$ is smooth iff
there is a neighborhood $U$ of $c$ in $\comps^d$ such that $\sing\cap U$ is a 
complex submanifold of $U$
\cite[Theorem 15, page 364]{BrKn1986}.

We will approximate $F_{n\alpha}$ with integrals and, in so doing, will need to 
consider the singularities of $F$ relevant to $\alpha$.
These singularities are called critical points.
For a more geometric explanation of their relevance and the stratified Morse 
theory behind it see \cite{PeWi2008}.

\begin{definition}
Let $\alpha\in\pnats^d$ and $c\in\sing$ be a smooth point.
We say that $c$ is \hl{critical} for $\alpha$ if it is a solution of the system
of $d-1$ equations
\[
  \alpha_1^{-1} x_1 \del_1 H(x)
  = \ldots = \alpha_d^{-1} x_d \del_d H(x). 
\]
We say $c$ is \hl{isolated} for $\alpha$ if, in addition, it has a 
neighbourhood in which it is the only critical point for $\alpha$.
\end{definition}

When $H$ is a polynomial, the system of $d$ equations in $d$ unknowns given by 
$H(x)=0$ and the critical point equations generally has a finite set of 
solutions.

\begin{remark}
At times it will be convenient to work in projective space.
Recall that $\cproj^{d-1}$ is the set of equivalences classes of 
$\comps^d\setminus\{0\}$ under the equivalence relation $\sim$ given by
$x\sim x'$ iff $x = \lambda x'$ for some $\lambda\in\comps\setminus\{0\}$.
Let $\bar{\; }:\comps^{d}\setminus\{0\}\to\cproj^{d-1}$ be the natural
map which takes a point $x$ to its equivalence class $[x]$.

Notice that the definition of critical point is well-defined for all 
$\alpha\in\cproj^{d-1}$ with nonzero components.
\end{remark}
\section{The full asymptotic expansion}\label{sec:asymptotics}  

Pemantle and Wilson \cite{PeWi2002} showed that if $\alpha\in\pnats^d$ and
$c\in\sing$ is strictly minimal, smooth, critical and isolated for $\alpha$, and
nondegenerate (which we will define shortly), then there exist $b_k\in\comps$ 
such that for all $N\in\pnats$ one has the asymptotic expansion 
\[ 
F_{n \alpha} 
= 
c^{-n \alpha} 
\left[ 
\sum_{k=0}^{N-1} 
b_k n^{-(d-1)/2 - k} +O\left(n^{-(d-1)/2-N}\right) \right]
\] 
as $n \to \infty$.
They also derived a similar expansion for degenerate points in the case $d=2$ 
and gave an explicit formula for $b_0$ for all $d$ when $p=1$.

In this section we derive explicit formulas for all $b_k$ and all $p$ and
prove that the asymptotic expansions are uniform in $\alpha$.

To formulate our results we employ the following functions.

\begin{definition}\label{functions}
Let $c\in\sing$ be smooth, and assume without loss of generality that 
$\del_d H(c)\neq 0$.
By the implicit function theorem there exists a bounded neighborhood $W$
of $\chat$ and a holomorphic function $h$ on $W$ such that $(w,h(w))\in\sing$, 
$\del_d H(w,h(w))\neq0$ for all $w\in W$, and $h(\chat) = c_d$.
Suppose $c_d \neq 0$, so that we may also assume that $h$ is nonzero on $W$.

For $0\le j<p$ define $u_j:W\to\comps$, $E:[-1,1]^{d-1}\to\comps^{d-1}$,
and $\utilde_j,\htilde:E^{-1}(W \cap C(c))\to\comps$ by
\begin{align*}
u_j(w)
  &= \lim_{y \to h(w)} (-y)^{-p+j}
  \frac{\del^j}{\del y^j} \Big( (y - h(w))^p F(w,y) \Big), \\
E(t)
  &= (c_1\me^{\mi t_1},\ldots,c_{d-1}\me^{\mi t_{d-1}}), \\
\utilde_j
  &= u_j\circ E \\
\htilde
  &= h \circ E.
\end{align*}
Furthermore, for $c$ critical for $\alpha$ define $\gtilde:E^{-1}(W)\to\comps$
by
\[
\gtilde(t)
  = \log \left( \frac{\htilde(t)}{\htilde(0)} \right) +
  \mi\sum_{m=1}^{d-1} \frac{\alpha_m}{\alpha_d} t_m.
\]

Then $\gtilde$ is well-defined for any $\alpha\in\cproj^{d-1}$ with 
$\alpha_d\neq 0$ since $\gtilde$ does not depend on the magnitude of $\alpha$
and since $h$ is nonzero on $W$.  
Moreover, $\utilde_j$, $\htilde$, and $\gtilde$ are all $C^\infty$ functions.

If $\det\gtilde''(0)\neq 0$, then $c$ is called \hl{nondegenerate}.
\end{definition}

In the context of a single pair of appropriate $c$ and $\alpha$ we will use the 
functions of Definition~\ref{functions} without further introduction.

Now for the first theorem.

\begin{theorem}\label{asymptotics}
Let $d\ge 2$ and $\alpha\in\pnats^d$.
If $c\in\sing$ is strictly minimal, smooth with $c_d\del_d H(c) \neq 0$, 
critical and isolated for $\alpha$, and nondegenerate, then for all $N\in\nats$,
\begin{align*}
F_{n \alpha} \tag{$\star$}
=
c^{-n \alpha } \Bigg[
& \Big( (2\pi\alpha_d n)^{d-1} \det \gtilde''(0) \Big)^{-1/2}
\sum_{j=0}^{p-1} \sum_{k=0}^{N-1} 
\frac{ (\alpha_d n+1)^{\overline{p-1-j}}}{(p-1-j)! j!}
(\alpha_d n)^{-k} L_k(\utilde_j, \gtilde) \\
 &
+O\left(n^{p-1-(d-1)/2-N}\right) 
\Bigg]
\end{align*}
as $n\to\infty$.

Here
\[
L_k(\utilde_j,\gtilde)
=  
\sum_{l=0}^{2k} 
\frac{\mathcal{H}^{l+k} (\utilde_j \gtildeunder^l)(t_0)}
  {(-1)^k 2^{l+k} l! (l+k)!},
\]
$\gtildeunder(t)= \gtilde(t)-\gtilde(t_0)-\frac{1}{2} (t-t_0) \gtilde''(t) (t-t_0)^T$,
$\mathcal{H}$ is the differential operator \\
$-\sum_{0\le r,s < d} (\gtilde''(t_0)^{-1})_{r,s} \del_r\del_s$, and $t_0=0$.
In every term of $L_k(\utilde_j,\gtilde)$ the total number of derivatives of
$\utilde$ and of $\gtilde''$ is at most $2k$.
\end{theorem}

In the case $d=2$ we can drop the nondegeneracy hypothesis.

\begin{theorem}\label{asymptotics_degenerate}
Let $d=2$ and $\alpha\in\pnats^d$. 
If $c\in \sing$ is strictly minimal, smooth with $c_d\del_d H(c) \neq 0$, 
critical and isolated for $\alpha$, and $v\ge 2$ is least such that 
$\gtilde^{(v)}(0)\neq 0$, then for all $N\in\nats$,
\begin{align*}
F_{n \alpha} \tag{$\dagger$}
=&
c^{-n \alpha } \Bigg[
\frac{(a \alpha_d n)^{-1/v}}{\pi v}
\sum_{j=0}^{p-1} \sum_{k=0}^{N-1} 
\frac{ (\alpha_d n+1)^{\overline{p-1-j}}}
{(p-1-j)!j!} (\alpha_d n)^{-2k/v} L^\text{even}_k(\utilde_j, \gtilde) \\
 &
+O\left(n^{p-1-(2N+1)/v}\right) \Bigg],
\end{align*}
as $n\to\infty$ for $v$ even and
\begin{align*}
F_{n \alpha} \tag{$\ddagger$}
=&
c^{-n \alpha } \Bigg[
\frac{(|a| \alpha_d n)^{-1/v}}{2\pi v}
\sum_{j=0}^{p-1} \sum_{k=0}^{N-1} 
\frac{ (\alpha_d n+1)^{\overline{p-1-j}}}
{(p-1-j)!j!} (\alpha_d n)^{-k/v} L^\text{odd}_k(\utilde_j, \gtilde) \\
 &
+O\left(n^{p-1-(N+1)/v}\right)
\Bigg],
\end{align*}
as $n\to\infty$ for $v$ odd.

Here
\begin{align*}
L^\text{even}_k(\utilde_j,\gtilde)
=&
\sum_{l=0}^{2k} 
\frac{(-1)^l \Gamma(\frac{2k+vl+1}{v})}{l!(2k+vl)!}
\left(a^{-1/v} \frac{d}{dt} \right)^{2k+vl} (\utilde_j\gtildeunder^l)(t_0), \\
L^\text{odd}_k(\utilde_j,\gtilde)
=&
\sum_{l=0}^k 
\frac{(-1)^l \Gamma(\frac{k+vl+1}{v})}{l!(k+vl)!}
\left( \zeta^{k+vl+1} +(-1)^{k+vl}\zeta^{-(k+vl+1)} \right) \\
 &
\times\left( |a|^{-1/v} \mi\sgn a \frac{d}{dt} \right)^{k+vl}
(\utilde_j\gtildeunder^l)(t_0),
\end{align*}
$\gtildeunder(t)= \gtilde(t) -\gtilde(t_0) -a (t-t_0)^v$,
$a= \gtilde^{(v)}(t_0)/v!$, $\zeta = \me^{\mi\pi/(2v)}$, and $t_0=0$.
In every term of $L^\text{even}_k(u,g)$ the total number of derivatives of $u$
and $g^{(v)}$ is at most $2k$, and in $L^\text{odd}_k(u,g)$ at most $k$.
\end{theorem}

To prove Theorems~\ref{asymptotics} and \ref{asymptotics_degenerate}
we follow the same general approach as in \cite{PeWi2002} and summarized in the
following steps: 
(1) use Cauchy's integral formula and strict minimality to express
$c^{n \alpha} F_{n \alpha}$ as a $d$-variate contour integral over a
contour almost touching $c$;
(2) expand the contour across $c$ and use Cauchy's residue theorem along with 
the smoothness of $c$ to express the innermost integral as a residue;
(3) calculate the residue explicitly, and take the resulting $(d-1)$-variate 
contour integral and change to real coordinates to get a Fourier-Laplace 
integral;
(4) use theorems from analysis (see Section~\ref{sec:analysis}) to 
approximate the integral asymptotically.


\begin{lemma}[{\cite[proof of Lemma 4.1]{PeWi2002}, Steps 1 \& 2}]\label{integral}  
Let $\alpha\in\pnats^d$ and $c\in \sing$.
If $c$ is strictly minimal and smooth with $c_d \del_d H(c) \neq 0$, 
then there exists $\epsilon\in(0,1)$ and a polydisc neighborhood $D$ of $\chat$ 
such that
\[
    c^{n\alpha} F_{n \alpha}
    =  
    c^{n\alpha} (2\pi\mi)^{1-d} 
    \int_{X} \frac{-R(w)}{w^{n \hat{\alpha}+1}}  dw
    + O\left(\epsilon^n\right)
\]
as $n\to\infty$, where  $X=D\cap C(\chat)$ and $R(w)$ is the residue of 
$y \mapsto F(w,y) y^{-\alpha_d n-1}$ at $h(w)$.
\end{lemma}

\begin{lemma}[Step 3]\label{FL_integral} 
In the previous lemma,
\[
R(w)
= 
-\sum_{j=0}^{p-1} 
\frac{(\alpha_d n + 1)^{\overline{p-1-j}}}{(p-1-j)! j!} 
h(w)^{-\alpha_d n} u_j(w).
\]
Thus
\[
c^{n \alpha}F_{n \alpha}
=
(2\pi)^{1-d}
\sum_{j=0}^{p-1} 
\frac{(\alpha_d n + 1)^{\overline{p-1-j}}}{(p-1-j)! j!}
\int_{\Xtilde}  \utilde_j(t) \me^{-\alpha_d n \gtilde(t)} dt 
+ O(\epsilon^n),
\]
as $n \to \infty$, where $\Xtilde = E^{-1}(X)$.
\end{lemma}

\begin{proof} 
This is a straightforward calculation.
Let $w\in X$.
Since $c$ is smooth, $h(w)$ is a simple zero of $y\mapsto H(w,y)$ and so a pole
of order $p$ of $y\mapsto F(w,y)$.
Since $h$ is nonzero on $W$, $h(w)$ is also a pole of order $p$ of 
$y\mapsto F(w,y) y^{-\alpha_d n -1}$.
Thus
\begin{align*}
R(w)
=&
\lim_{y \to h(w)} \frac{1}{(p-1)!} 
\left( \frac{\del}{\del y} \right)^{p-1}
\Big( (y - h(w))^p F(w,y) y^{-\alpha_d n -1} \Big) \\
=&
-\frac{1}{(p-1)!}\lim_{y \to h(w)} \sum_{j=0}^{p-1} 
\binom{p-1}{j} 
\left( \frac{\del}{\del y} \right)^j \Big( (y - h(w))^p \, F(w,y) \Big)\\
 &
\times
(\alpha_d n + 1)^{\overline{p-1-j}}  (-1)^{p-j} y^{-\alpha_d n -p +j} \\
=&
-\sum_{j=0}^{p-1}
 \frac{(\alpha_d n + 1)^{\overline{p-1-j}}}{(p-1-j)!j!}
h(w)^{-\alpha_d n} \lim_{y \to h(w)} (-y)^{-p+j} \left( \frac{\del}{\del y} \right)^j
\Big( (y - h(w))^p \, F(w,y) \Big),
\end{align*}
from which the first identity follows by definition of $u_j$.
Thus
\begin{align*}
 &
c^{n \alpha} (2\pi\mi)^{1-d}
\int_{X} \frac{-R(w)}{w^{n \hat{\alpha}+1}} dw \\
=&
c^{n \alpha} (2\pi\mi)^{1-d} \int_{X} \frac{1}{w^{n\alphahat  + 1}}
\sum_{j=0}^{p-1} 
\frac{(\alpha_d n + 1)^{\overline{p-1-j}}}{(p-1-j)! j!}
h(w)^{- n \alpha_d} u_j(w) dw \\
=&
(2\pi\mi)^{1-d}
\sum_{j=0}^{p-1}
\frac{(\alpha_d n + 1)^{\overline{p-1-j}}}{(p-1-j)! j!} 
\int_{X} \frac{\chat^{n \alphahat}}{w^{n \alphahat}} 
u_j(w) \left( \frac{h(w)}{h(\chat)} \right)^{-\alpha_d n} 
\frac{dw}{\prod_{m=1}^{d-1} w_m} \\
=&
(2\pi)^{1-d}
\sum_{j=0}^{p-1} 
\frac{(\alpha_d n + 1)^{\overline{p-1-j}}}{(p-1-j)! j!} 
\int_{\Xtilde} \prod_{m=1}^{d-1} \me^{-\mi \alpha_m n t_m} 
\utilde_j(t) \left( \frac{\htilde(t)}{\htilde(0)} \right)^{-\alpha_d n} dt \\
 & 
\text{(via the change of variables $w = E(t)$)} \\
=&
(2\pi)^{1-d}
\sum_{j=0}^{p-1}
\frac{(\alpha_d n + 1)^{\overline{p-1-j}}}{(p-1-j)! j!}
\int_{\Xtilde} \utilde_j(t) \,\me^{-\alpha_d n \gtilde(t)} dt,
\end{align*}
which with Lemma~\ref{integral} proves the stated formula for 
$c^{n\alpha}F_{n\alpha}$.
\end{proof}

\begin{remark}
In the case $d=1$ Lemma~\ref{FL_integral} simplifies:
$h$, $u_j$, and $R$ become 0-ary functions, that is, constants ($h$ becomes $c$),
and there is no integral.
Thus we arrive at the following known result.  

If $c\in\sing$ is strictly minimal and smooth, then there exists 
$\epsilon\in(0,1)$ such that
\[
F_n 
= 
c^{-n}\left[
\sum_{j=0}^{p-1} 
\frac{(n + 1)^{\overline{p-1-j}}}{(p-1-j)! j!} u_j 
+ O\left(\epsilon^n\right)
\right],
\]
as $n\to\infty$, where
\[
u_j = \lim_{x\to c} (-x)^{-p+j} \left( \frac{\del}{\del x} \right)^j
\Big( (x-c)^p F(x) \Big).
\]
Moreover, if $c\in\sing$ is finitely minimal and smooth and every point of 
$\sing\cap C(c)$ is smooth, then the asymptotic expansion of $F_n$ is the sum of 
the expansions around each point of $\sing\cap C(c)$.
\end{remark}

Before proceeding to Step 4 we collect a few technical facts about $\gtilde$.

\begin{lemma}\label{gtilde}
Let $\alpha\in\pnats^d$ and $c\in\sing$.
If $c$ is strictly minimal, smooth with $c_d \del_d H(c)\neq 0$, 
and critical and isolated for $\alpha$, then for all $t\in\Xtilde$ we have 
$\Re\gtilde(t) \geq 0$ with equality only at $t=0$ 
and $\gtilde'(t) = 0$ iff $t=0$.
\end{lemma}

\begin{proof}
Firstly, $\gtilde(0)=0$ by definition.
Secondly, $\Re\gtilde(t)\geq 0$ with equality only at 0 since $c$ is strictly 
minimal.
Lastly, by the implicit function theorem, 
$\del_m h(w) = -\del_m H(w,h(w)) / \del_d H(w,h(w))$
for all $m<d$ and $w \in W$.
So for $t\in\Xtilde$ we have
\[
\del_m \gtilde(t)
= 
-\mi c_m \me^{\mi t_m} \frac{1}{h(E(t))} 
\frac{\del_m H(E(t),h(E(t)))}{\del_d H(E(t),h(E(t)))} +\mi \frac{\alpha_m}{\alpha_d}.
\]
Therefore $\gtilde'(t) = 0$ iff 
\[
\alpha_m^{-1} \, c_m \, \me^{\mi \, t_m} \, \del_m H(E(t),h(E(t)))  
= \alpha_d^{-1} \, h(E(t)) \, \del_d H(E(t),h(E(t)))
\]
for all $m<d$ iff $(E(t),h(E(t)))$ is critical for $\alpha$
iff $t=0$ since $c$ is isolated for $\alpha$.
\end{proof}

\begin{proof}[Proof of Theorem~\ref{asymptotics} (Step 4)]
By Lemmas \ref{integral} and \ref{FL_integral} there exists $\epsilon\in(0,1)$ 
and an open bounded neighbourhood $\Xtilde$ of $0$ such that
\[
c^{n \alpha}F_{n \alpha}
= 
(2\pi)^{1-d} 
\sum_{j=0}^{p-1} 
\frac{(\alpha_d n + 1)^{\overline{p-1-j}}}{(p-1-j)! j!} I_{j,n}
+O\left(\epsilon^n\right)
\]
as $n\to\infty$, where 
$I_{j,n} = \int_{\Xtilde}  \utilde_j(t) \me^{-\alpha_d n \gtilde(t)} dt$.

Choose $\kappa \in C_c^\infty(\Xtilde)$ such that $\kappa = 1$ on a 
neighbourhood $Y$ of $0$.
Then
\[
I_{j,n} 
= 
\int_{\Xtilde} \kappa(t) \utilde_j(t) \me^{-\alpha_d n \gtilde(t)} dt
+\int_{\Xtilde} (1-\kappa(t)) \utilde_j(t) \me^{-\alpha_d n\gtilde(t)}dt.
\] 
The second integral decreases exponentially as $n\to\infty$ since $\Re\gtilde$ 
is strictly positive on the compact set $\cl{\Xtilde \setminus Y}$ by 
Lemma~\ref{gtilde}.
By Lemma~\ref{gtilde} again and our nondegeneracy hypothesis, we we may apply 
Theorem~\ref{Hormander} with $t_0=0$ to the first integral.
Noting that $L_k(\kappa\utilde_j, \gtilde) = L_k(\utilde_j, \gtilde)$ because 
the derivatives are evaluated at $0$ and $\kappa=1$ in a neighbourhood of $0$,
this gives 
\begin{align*}
I_{j,n}
&=
\me^{-n_d \, \gtilde(0)}
\left( \det\left(\frac{\alpha_d n \, \gtilde''(0)}{2\pi}\right)\right)^{-1/2}
\sum_{k=0}^{N-1} 
(\alpha_d n)^{-k} L_k (\utilde_j, \gtilde)
+ O( (\alpha_d n)^{-(d-1)/2-N}) \\
&=
\left(\frac{\alpha_d n}{2\pi}\right)^{-(d-1)/2}
\left(\det \gtilde''(0) \right)^{-1/2}
\sum_{k=0}^{N-1}
(\alpha_d n)^{-k} L_k(\utilde_j, \gtilde)
+ O( n^{-(d-1)/2-N})
\end{align*}
as $n \to \infty$.
Hence
\begin{align*}
c^{n \alpha}F_{n \alpha}
=&
(2\pi)^{1-d}
\sum_{j=0}^{p-1}
\frac{(\alpha_d n + 1)^{\overline{p-1-j}}}{(p-1-j)! j!} I_{j,n} 
+O\left(\epsilon^n\right) \\
=&
\sum_{j=0}^{p-1} \sum_{k=0}^{N-1}
\frac{ (\alpha_d n+1)^{\overline{p-1-j}}}{(p-1-j)! j!}
\Big( (2\pi \alpha_d n)^{d-1} \det \gtilde''(0) \Big)^{-1/2} 
(\alpha_d n)^{-k} L_k(\utilde_j, \gtilde) \\
 &
+ O\left(n^{p-1-(d-1)/2-N}\right),
\end{align*}
as $n \to \infty$, as desired.
\end{proof}

The proof of Theorem~\ref{asymptotics_degenerate} is similar but uses
Theorem~\ref{Elst} instead of Theorem~\ref{Hormander}.

In the case of finitely minimal points of $\sing$ we simply take an open set $W$
for each finitely minimal point so that $W$ contains no other finitely 
minimal points, repeat the proofs above for each such $W$, and add the 
asymptotic expansions.
Thus we have the following.

\begin{corollary}
Let $\alpha\in\pnats^d$ and $c\in \sing$.
If $c$ is finitely minimal and every point of $\sing\cap C(c)$
satisfies the hypotheses (excluding strict minimality) of 
Theorem~\ref{asymptotics} or Theorem~\ref{asymptotics_degenerate},
then the asymptotic expansion of $F_{n \alpha}$ equals the sum of the 
expansions around each point of $\sing\cap C(c)$.
\end{corollary}

Finally, we show that the asymptotic formulas for $F_{n\alpha}$ are uniform in
$\alpha$.
This was claimed without proof in \cite{PeWi2002}.

\begin{proposition}\label{uniform}
Let $d\ge 2$ and $K\subseteq\rproj^{d-1}$ be compact. 
Suppose that for all $\alpha\in K$ there exists a unique $c\in\sing$ that is 
strictly minimal, smooth with $c_d \del_d H(c)\neq 0$, critical and isolated for
$\alpha$, and nondegenerate.
Suppose that all these points $c$ lie in a bounded open set  
$V\subset\sing$ such that $x_d \del_d H(x)\neq 0$ for all $x\in V$.
Then for each $N\in\nats$ the big-oh constant of ($\star$) stays bounded as 
$\alpha\in\pnats$ varies and $\alphabar$ stays in $K$.
\end{proposition}

\begin{proposition}\label{uniform_degenerate}
Let $d=2$ and $K\subseteq\rproj^1$ be compact. 
Suppose there exists $v\ge 2$ such that for all $\alpha\in K$ there exists a 
unique $c\in\sing$ that is strictly minimal, smooth with 
$c_d \del_d H(c)\neq 0$, critical and isolated for $\alpha$, and $v$ is the 
least integer greater than 2 such that $g^{(v)}(0)\neq 0$.
Suppose that all these points $c$ lie in a bounded open set  
$V\subset\sing$ such that $x_d \del_d H(x)\neq 0$ for all $x\in V$.
Then for each $N\in\nats$ the big-oh constant of ($\dagger$) and 
($\ddagger$) stays bounded as $\alpha\in\pnats$ varies and $\alphabar$ stays
in $K$. 
\end{proposition}

\begin{proof}[Proof of Proposition~\ref{uniform}]
We take up where the proof of Theorem~\ref{asymptotics} left off.
Recall that $\gtilde$ depends on $\alpha$.
Let us emphasize this dependence by writing $\gtilde_\alpha$.
By Theorem~\ref{Hormander} it suffices to show that for any fixed positive 
integer $P$ there exists $M>0$ such that for all $\alpha\in K$ and all 
$\beta\in\nats^{d-1}$ with $|\beta|\le P$ we have 
\begin{equation}\label{supnorm}
||\del^\beta \gtilde_\alpha ||_\infty 
=  
\sup \{ |\del^\beta \gtilde_\alpha(t)| : t\in\Xtilde\} \le M,
\end{equation}
where $\Xtilde\subset\reals^{d-1}$ is a suitable neighborhood of 0.

To prove such a bound we first show that the correspondence between a direction
in $K$ and its critical point is continuous.
To this end it will be helpful to introduce the logarithmic Gauss map 
$\gamma:V\to\cproj^{d-1}$ defined by 
$x\mapsto [x_1 \del_1 H(x),\ldots,x_d \del_d H(x)]$.
Note that $\gamma$ is well-defined since 
$(x_1 \del_1 H(x),\ldots,x_d \del_d H(x))\neq 0$ on $V$ by hypothesis.
Note also that $\alpha\in K$ is critical for $c$ iff 
$\gamma(c) = \alpha$.

Since $H$ is holomorphic on $V$, each $\del_j H$ is holomorphic on $V$, so 
that $\gamma$ is continuous.
By hypothesis, for every $\alpha\in K$ the preimage $\gamma^{-1}(\{\alpha\})$ 
contains exactly one element.
Thus $\gamma^{-1}:K\to V$ is well-defined.
Moreover $\gamma^{-1}$ is continuous since $\gamma$ restricted to 
$\gamma^{-1}(K)$ is a continuous bijection on a compact space into a Hausdorff
space and therefore a homeomorphism.

Each $\alpha\in K$ has associated to it a point 
$\gamma^{-1}(\alpha)\in\gamma^{-1}(K)$ which has has associated to it the 
functions of Definition~\ref{functions}.
We now show that finitely many such functions will do to handle all points of 
$\gamma^{-1}(K)$.
For each $c\in\gamma^{-1}(K)$ let $D$ be the polydisc from 
Lemma~\ref{integral} containing $\chat$ and $h$ the nonzero holomorphic function
on $D$ associated to $c$.
Let $D'$ be a polydisc such that $D'\subset \cl{D'} \subset D$.
The collection of all such $D'\times h(D')$ forms an open cover for the 
compact space $\gamma^{-1}(K)$.
So this cover has a finite subcover 
$D'_1\times h_1(D'_1),\ldots,D'_l\times h_l(D'_l)$.
Let $D_1,\ldots,D_l$ be the superpolydiscs corresponding to $D'_1,\ldots,D'_l$.
To handle the various $E$ of Definition~\ref{functions}, define 
$e:\prod_{j=0}^l \cl{D'_j} \times[-1,1]^{d-1}\to\comps^{d-1}$ by 
$(w,t)\mapsto (w_1\me^{\mi t_1},\ldots,w_{d-1}\me^{\mi t_{d-1}})$.
Then $e$ is continuous, and for all $j$ the open set $e^{-1}(D_j)$ contains
$\cl{D'_j}\times\{0\}$.
Since $\cl{D'_j}$ is compact, by the tube lemma \cite[Lemma 5.8]{Munk1975} there
exists a neighborhood $Y_j$ of 0 in $[-1,1]^{d-1}$ such that
$e^{-1}(D_j) \supseteq \cl{D'_j}\times Y_j$.  
Set $Y = \bigcap_{j\le l} Y_j \subset [-1,1]^{d-1}$, and let $\Xtilde$ be a 
neighborhood of 0 with $\Xtilde\subset\cl{\Xtilde}\subset Y$.
Thus to each $\alpha\in K$ are associated some $c= \gamma^{-1}(\alpha)$, $D'_j$, 
$h_j$, $\htilde_j$, and
$\log \left( \frac{\htilde_j(t)}{\htilde_j(0)} \right) +
\mi\sum_{0\le m<d} \frac{\alpha_m}{\alpha_d} t_m = \gtilde_\alpha$, where the last 
two functions are defined on $Y_j$ and so on $Y$.

With this setup we now show \eqref{supnorm}.
Since $A_j: Y\to\reals$ for $j\le l$ and $B:K\times Y\to\reals$ defined by
$A_j(t)= \Big| \log\left(\tfrac{\htilde_j(t)}{\htilde_j(0)}\right) \Big|$ and
$B(\alpha,t)= \Big| \sum_{0\le m<d} \tfrac{\alpha_m}{\alpha_d} t_m \Big|$ are
continuous and $\cl{\Xtilde}$ and $K$ are compact, we have that 
$M_A:= \sum_{0\le j\le l} \max\{A_j(t) : t\in\cl{\Xtilde}\} < \infty$,
$M_B:= \max\{ B(\alpha,t) : \alpha\in K, t\in\cl{\Xtilde} \} < \infty$, and
\[
||\gtilde_\alpha||_\infty 
\le     
\sup_{\alpha\in K, t\in\Xtilde} A_j(t) +B(\alpha,t) 
\quad\text{(for some $j\le l$)}\quad 
\le  
M_A +M_B,
\]
a bound that is independent of $\alpha$.
Similarly, since each $\htilde_j$ is $C^\infty$ over $Y$, each  
$||\del^\beta g_\alpha ||_\infty$ for $|\beta|\le P$ stays bounded as $\alpha$
varies within $K$, as desired. 
\end{proof}

The proof of Theorem~\ref{uniform_degenerate} is similar but uses
Theorem~\ref{Elst} instead of Theorem~\ref{Hormander}.
\section{Rewriting the expansion in terms of the original data}
\label{sec:original}

To actually compute with the formulas in Theorems~\ref{asymptotics} and 
\ref{asymptotics_degenerate} it is helpful to rewrite the quantities involved
in terms of the original data $G$ and $H$.
The propositions below give formulas for calculating $\gtilde''(0)$ and $u_j$ 
in terms of derivatives of $G$ and $H$.

\begin{proposition}[{\cite[Theorem 3.3]{RaWi2007}}]\label{hessian}
Let $\alpha\in\pnats^d$ and $c\in\sing$.
If $c$ is smooth with $c_d \del_d H(c)\neq 0$ and critical for $\alpha$, 
then for all $l,m < d$ with $l\neq m$ we have
{\small
\begin{align*}
\gtilde''(0)_{lm}
=&
\frac{c_l c_m}{c_d^2 (\del_d H)^2} \\
 &
\times \left( \del_m H \del_l H+ c_d
(\del_d H \del_m \del_l  H- \del_m H \del_d \del_l H
-\del_l H \del_m \del_d  H+ \frac{\del_l H \del_m H}{\del_d H}
\del_d^2  H)
\right) \Bigg|_{x = c} \\
\gtilde''(0)_{ll}
=&
\frac{c_l \del_l H}{c_d \del_d H} + \frac{c_l^2}{c_d^2 (\del_d H)^2}
\left( (\del_l H)^2+ c_d
(\del_d H \del_l^2 H- 2 \del_l H \del_d \del_l H
+ \frac{(\del_l H)^2}{\del_d H} \del_d^2 H)
\right) \Bigg|_{x = c}.
\end{align*}
}
\end{proposition}

In the presence of symmetry Proposition~\ref{hessian} simplifies greatly.

\begin{proposition}[{\cite[Proposition 3.4]{RaWi2007}}]\label{symmetric}
Let $\alpha\in\pnats^d$ and $c\in\sing$.
If $x \mapsto H(x)$ is symmetric, $\alpha$ has all of its components equal, and 
$c$ lies in the positive orthant, is strictly minimal, smooth with 
$c_d \del_d H(c)\neq 0$, and critical for $\alpha$, then $c$ has all of its 
components equal and for all $l,m < d$ with $l\neq m$,
\[
\gtilde''(0)_{lm} = q,
\quad
\gtilde''(0)_{ll} = 2q, 
\quad
\text{and}
\quad
\det\gtilde''(0) =  d q^{d-1},
\]
where $q = 1 + \frac{c_1}{\del_d H}(\del_d^2 H - \del_1 \del_d H) \Big|_{x=c}$.
\end{proposition}

\begin{proposition}
If $c\in\sing$ is smooth with $c_d \del_d H(c)\neq 0$, then for all $j<p$ 
and $w\in W$ we have
\[
u_j(w)
=
(-h(w))^{-p+j} \lim_{y \to h(w)} \frac{\partial^j}{\partial y^j} 
\frac{G(w,y)}{Q(w,y)^p},
\]
where $Q:((W\times\comps)\cap\mathcal{D})\setminus\sing\to\comps$ is given by
$Q(w,y) = \frac{H(w,y)}{y - h(w)}$.
Moreover, 
\[
\lim_{y \to h(w)} \del_d^j Q(w,y) 
= \frac{1}{j+1} \del_d^{j+1} H(w,h(w)). 
\]
In particular,
\[
u_0(w) = \frac{G(w,y)}{\Big(-h(w) \del_d H(w,h(w))\Big)^p}.
\]
\end{proposition}

\begin{proof}
The first statement is just the definition of $u_j$.

Let $(w,y) \in ((W \times \comps) \cap \mathcal{D}) \setminus \sing$.
Then $H(w,h(w)) = 0$ and
\begin{align*}
Q(w,y)
&=
\frac{H(w,y)}{y - h(w)} \\
&=
\frac{H(w,y)}{y - h(w)}
\sum_{n=1}^\infty \frac{\del_d^n H(w,h(w))}{n!} (y-h(w))^n \\
&=
 \sum_{n=0}^\infty \frac{\del_d^{n+1} H(w,h(w))}{(n+1)!} (y-h(w))^n.
\end{align*}
Therefore $Q$ extends to a holomorphic function on 
$(W\times\comps)\cap\mathcal{D}$ and 
\begin{align*}
\lim_{y\to h(w)} \del_d^j Q(w,y)
=& 
\lim_{y\to h(w)} \sum_{n=j}^\infty n(n-1)\cdots(n-j+1)
\frac{\del_d^{n+1} H(w,h(w))}{(n+1)!} (y-h(w))^{n-j} \\
=& 
\frac{1}{j+1} \del_d^{j+1} H(w,h(w)) .
\end{align*}
\end{proof}
\section{Examples}\label{sec:examples}

Let us apply the results of Section~\ref{sec:asymptotics} and 
\ref{sec:original} to a few representative combinatorial examples.
We used Maple 11 to do the calculations, and our worksheets are available at
\url{http://www.cs.auckland.ac.nz/~raichev/research.html}.

First we mention two shortcuts to finding strictly minimal points.

\begin{proposition}[{\cite[Theorem 3.16]{PeWi2008}}]\label{minimal}
If the coefficients of $F$ are all nonnegative and there is a critical point for
$\alpha$, then there is a minimal critical point for $\alpha$ in $\preals^d$.
\end{proposition}

A $d$-variate power series $\sum a_\alpha x^\alpha$ is called \hl{aperiodic} if
the $\ints$-span of $\{\alpha\in\nats^d : a_\alpha \neq 0\}$ equals $\ints^d$.

\begin{proposition}[{\cite[Proposition 3.17]{PeWi2008}}]\label{strictly_minimal}
If $1 - H$ is aperiodic and has nonnegative coefficients, then every minimal 
point of $\sing$ is strictly minimal and lies in $\preals^d$.
\end{proposition}

\begin{example}[$d=2$, $p=1$, $N=2$]
Consider the bivariate generating function
\[
  F(x_1,x_2) = \frac{1}{1- x_1- x_2- x_1x_2}
\]
whose coefficients $F_{\beta_1,\beta_2}$ are called Delannoy numbers and count 
the number of lattice paths from $(0,0)$ to $(\beta_1,\beta_2)$ with allowable 
steps $(1,0)$, $(0,1)$ and $(1,1)$.
We compute the first two terms of the asymptotic expansion of $F_{n\alpha}$
as $n\to\infty$ for $\alpha=(3,2)$.

The critical points of $\sing$ are
\[
(-\tfrac{2}{3}+\tfrac{1}{3} \sqrt{13}, -\tfrac{3}{2}+\tfrac{1}{2} \sqrt{13})
\quad\text{and}\quad
(-\tfrac{2}{3}-\tfrac{1}{3} \sqrt{13}, -\tfrac{3}{2}-\tfrac{1}{2} \sqrt{13}).
\]
Both points are smooth and the first point, which we denote by $c$,
is strictly minimal by Propositions~\ref{strictly_minimal} and \ref{minimal}.

Applying the results of Sections~\ref{sec:asymptotics} and \ref{sec:original}, 
we get
\[
  F_{n \alpha}
  =
 \left( c_1^{-3} c_2^{-2} \right)^n 
 \left( b_0 n^{-1/2} +b_1 n^{-3/2} +O\left(n^{-5/2}\right)\right)
\]
as $n\to\infty$, where $c_1^{-3} c_2^{-2} \approx 71.16220050$, 
$b_0 = \frac{13^{3/4} \sqrt{3}}{156 \sqrt{\pi} } (5+\sqrt{13})
  \approx 0.3690602772 $ and
$b_1 = -\frac{5\cdot 13^{3/4} \sqrt{3}}{1898208 \sqrt{\pi}} (79\sqrt{13}+767)
  \approx -0.01853610557 $.

Comparing this approximation with the actual values of
$F_{n \alpha}$ for small $n$ (using 10-digit floating-point arithmetic),
we get the following table.
\vspace{1em}

\begin{center}
{\tiny
\begin{tabular}{|l|lllll|}
\hline
$n$ & 1 & 2 & 4 & 8 & 16 \\
\hline
$F_{n \alpha}$
    & 25            & 1289          & 4.673345$\cdot 10^6$
    & 8.527550909$\cdot 10^{13}$    & 3.978000114$\cdot 10^{28}$ \\
$c^{-n \alpha}(b_0 n^{-1/2})$
    & 26.26314145   & 1321.542224   & 4.732218447$\cdot 10^6$ 
    & 8.581184952$\cdot 10^{13}$    & 3.990499094$\cdot 10^{28}$ \\
$c^{-n \alpha}(b_0 n^{-1/2} +b_1 n^{-3/2})$
    & 24.94407138   & 1288.354900  & 4.672799360$\cdot 10^6$ 
    & 8.527311037$\cdot 10^{13}$   & 3.977972633$\cdot 10^{28}$ \\
one-term relative error
    & -0.05052565800  & -0.02524610085 & -0.01259771042
    & -0.006289501355 & -0.003142026054 \\
two-term relative error
    & 0.002237144800  & 0.0005004654771  & 0.0001167557713
    & 0.00002812906104 & 0.000006908245151 \\
\hline
\end{tabular}
}
\end{center}
\vspace{1em}

For an arbitrary $\alpha$, the two-term asymptotic expansion of
$F_{n \alpha}$ is just as easy to compute symbolically in $\alpha$.
The corresponding constants $c_1$, $c_2$, $a_1$, $a_2$ are
square roots of rational functions of $\alpha_1$, $\alpha_2$, and
$\sqrt{\alpha_1^2 + \alpha_2^2}$.
The exact formulas are somewhat long, so we omit them.
\end{example}

\begin{example}[$d=2$, $p=1$, $N=5$]
Fix $q\in(0,1)$ and consider the bivariate generating function
\[
F(x_1,x_2) = \frac{1 -q x_1}{1 -q x_1 +q x_1 x_2 -x_1^2 x_2}
\]
which arises in the context of quantum random walks.
Motivated by an example from \cite{BrPe2007}, we compute the asymptotics of
$F_{n\alpha}$ as $n\to\infty$ for $\alpha=(2,1-q)$.

There is one critical point of $\sing$ for $\alpha$, namely, $c:= (1,1)$.
This point is smooth.
Explicitly solving for $x_2$ as a function of $x_1$ in $H(x_1,x_2)=0$ 
and applying the minimum modulus theorem shows that $c$ is minimal.
However, $c$ is not finitely minimal, because for every $p_1\in C(1)$ there 
exists $p_2\in C(1)$ such that $(p_1,p_2)\in\sing$.
By a modification of Lemma~\ref{FL_integral} described in 
\cite[Theorem 3.2]{BrPe2007} we may still apply Theorem~\ref{asymptotics_degenerate}
to $c$.
Doing so, we get
\[
F_{n \alpha} = b_0 \, n^{-1/3} +b_4 \, n^{-5/3} + O\left( n^{-2}\right)
\]
as $n\to\infty$, where 
$b_0=\frac{ (1-q)^{2/3} }{ 3^{2/3} q^{1/3}(1+q)^{1/3} \Gamma(2/3) }$
and $b_4=  -\frac{3^{1/6} (q^4+22q^2+1)\Gamma(2/3)}
{280 q^{5/3}(q+1)^{5/3} (1-q)^{2/3}\pi }$.
The $n^{-2/3}$, $n^{-1}$, and $n^{-4/3}$ terms are zero.

Comparing this approximation with the actual values of
$F_{n \alpha}$ for small $n$ (using 10-digit floating-point arithmetic),
we get the following table for $q=1/2$.
\vspace{1em}

\begin{center}
{\tiny
\begin{tabular}{|l|lllll|}
\hline
$n$ & 2 & 4 & 8 & 16 & 32 \\
\hline
$F_{n \alpha}$
    & 0.1875000000  & 0.1523437500  & 0.1221771240
    & 0.09739671811 & 0.07744253816 \\
$b_0 n^{-1/3}$
    & 0.1953794677  & 0.1550727862  & 0.1230813520 
    & 0.09768973380 & 0.07753639314 \\
$b_0 n^{-1/3} + b_4 n^{-5/3}$
    & 0.1855814246  & 0.1519865960  &  0.1221092630
    & 0.09738354495 & 0.07743994970 \\
one-term relative error
    &-0.04202382773 &-0.01791367352 &-0.007400959937
    &-0.003008476011 &-0.001211930578  \\
two-term relative error
    & 0.01023240213 & 0.002344395487  & 0.0005554313097
    & 0.0001352526066 & 0.00003342426606  \\
\hline
\end{tabular}
}
\end{center}
\vspace{1em}
\end{example}

\begin{example}[$d=3$, $p\le 3$, $N=2$]
Consider the $(d+1)$-variate generating function
\[
  W(x_1,\ldots,x_d,y) = \frac{A(x)}{1 - y \, B(x)},
\]
where $A(x) = 1/ \big( 1 - \sum_{j=1}^d \frac{x_j}{x_j+1} \big)$,
$B(x) = 1- (1- e_1(x)) A(x)$, and $e_1(x) = \sum_{i=j}^d x_j$.
Using the symbolic method (as presented in \cite[Chapter 1]{FlSe2009}, say)
it is not difficult to show that $W$ counts all words over a $d$-ary 
alphabet $\Lambda$, where $x_j$ marks occurrences of letter $j$ of $\Lambda$ and
$y$ marks occurrences of \emph{snaps}, non-overlapping pairs of duplicate 
letters counted from left to right.
Here $A(x)$ counts snapless words over $\Lambda$, the so-called Smirnov words.
For more details see \cite{RaWic}.

The coefficient $W_{(n,\ldots,n,s)}$ is then the number of words with $n$ 
occurrences of each letter and $s$ snaps.
For $n\in\pnats$ let $\psi_n$ be the random variable taking a word over $\Lambda$
with $n$ occurrences of each letter and returning the number of snaps in the 
word.
We compute the expectation and variance of $\psi_n$ as $n\to\infty$.
If $\alpha = (1,\ldots,1)$, then
\begin{align*}
\mathbb{E}(\psi_n)
&=
\frac{ \Big( \frac{\del W}{\del y}(x,1) \Big)_{n\alpha}}
{\Big( W(x,1) \Big)_{n\alpha}} 
=
\frac{ \Big( A(x)^{-1} B(x)(1 - e_1(x))^{-2} \Big)_{n\alpha}}
{\Big( (1 - e_1(x))^{-1} \Big)_{n\alpha}} \\
\mathbb{E}(\psi_n^2)
&=
\frac{ \Big( 
\frac{\del^2 W}{\del y^2}(x,1) + \frac{\del W}{\del y}(x,1)
\Big)_{n\alpha} }
{ \Big( W(x,1) \Big)_{n\alpha}} \\
&=
\frac{ \Big( A(x)^{-2}B(x)(B(x)+1)(1-e_1(x))^{-3} \Big)_{n\alpha}}
{ \Big( (1-e_1(x))^{-1} \Big)_{n\alpha}} \\
\mathbb{V}(\psi_n)
&=
\mathbb{E}(\psi_n^2) -\mathbb{E}(\psi_n)^2.
\end{align*}


Let $H(x) = 1-e_1(x)$.
Then the only critical point of $\sing$ is $c: = (1/d,\ldots,1/d)$,
and it is strictly minimal by Propositions~\ref{minimal} and
\ref{strictly_minimal}.
Applying the results of Sections~\ref{sec:asymptotics} and \ref{sec:original} to
$F_1(x):= W(x,1)$ (with $p=1$), $F_2(x):= \del W/\del y (x,1)$ (with $p=2$), and
$F_3(x):=  \del^2 W/\del y^2 (x,1) + \del W/\del y (x,1)$ (with $p=3$) 
with $d=3$, we obtain
\begin{align*}
\mathbb{E}(\psi_n)
&= 
\frac{ [(2\pi n)^2 d]^{-1/2} 
\sum_{j=0}^1 \sum_{k=0}^1
\frac{(n+1)^{\overline{1-j}}}{(1-j)!j!} n^{-k} L_k(\utilde_{2,j}, \gtilde)
+ O\left(n^{-2}\right) }
{ [(2\pi n)^2 d)]^{-1/2} 
\sum_{j=0}^0 \sum_{k=0}^1
\frac{(n+1)^{\overline{-j}}}{(-j)j!} n^{-k} L_k(\utilde_{2,j}, \gtilde) 
+ O\left(n^{-3}\right) } \\
&=
\frac{ \frac{3\sqrt{3}}{8\pi} -\frac{61\sqrt{3}}{192\pi}n^{-1} +O\left(n^{-2}\right) }
{ \frac{\sqrt{3}}{2\pi}n^{-1} -\frac{\sqrt{3}}{9\pi}n^{-2} +O\left(n^{-3}\right) } \\
&= 
\tfrac{3}{4} n - \tfrac{15}{32}  +O\left(n^{-1}\right) ,  \\
\mathbb{E}(\psi_n^2)
&=
\frac{ [(2\pi n)^2 d]^{-1/2} 
\sum_{j=0}^2 \sum_{k=0}^1
\frac{(n+1)^{\overline{2-j}}}{(2-j)!j!} n^{-k} L_k(\utilde_{3,j},\gtilde)
+ O\left(n^{-1}\right) }
{ [(2\pi n)^2 d]^{-1/2} 
\sum_{j=0}^0 \sum_{k=0}^1
\frac{(n+1)^{\overline{-j}}}{(-j)j!} n^{-k} L_k(\utilde_{1,j}, \gtilde) 
+ O\left(n^{-3}\right) } \\
&=
\frac{ \frac{9\sqrt{3}}{32\pi}n -\frac{35\sqrt{3}}{128\pi} +O\left(n^{-1}\right) }
{ \frac{\sqrt{3}}{2\pi}n^{-1} -\frac{\sqrt{3}}{9\pi}n^{-2} +O\left(n^{-3}\right) }\\
&=
\tfrac{9}{16} n^2 -\tfrac{27}{64} n +O(1), \quad\text{and} \\
\mathbb{V}(\psi_n)
&=
\tfrac{9}{32} n +O(1).
\end{align*}

Comparing these approximations with the actual values for small $n$ 
(using 10-digit floating-point arithmetic), we get the following table.
\vspace{1em}

\begin{center}
{\tiny
\begin{tabular}{|l|lll|}
\hline
$n$  & 2 & 4 & 8  \\
\hline
$\mathbb{E}(\psi_n)$
     & 1.000000000 & 2.509090909 & 5.520560294   \\
$(3/4)n$
     & 1.500000000 & 3 & 6 \\
$(3/4)n -15/32 $
     & 1.031250000 & 2.531250000 & 5.531250000   \\
one-term relative error
     & -0.5000000000 & -0.1956521740 & -0.08684620409  \\
two-term relative error
     & -0.03125000000 & -0.008831521776 & -0.001936344398   \\
\hline
$\mathbb{E}(\psi_n^2)$
     & 1.800000000  & 7.496103896 & 32.79620569   \\
$(9/16)n^2$
     & 2.250000000 & 9 & 36 \\
$(9/16)n^2 -(27/64)n$
     & 1.406250000 & 7.312500000 & 32.62500000   \\
one-term relative error
     & -0.2500000000 & -0.2006237006 & -0.09768795635 \\
two-term relative error
     & 0.2187500000  & 0.02449324323 & 0.005220289555   \\
\hline
$\mathbb{V}(\psi_n)$
     & 0.800000000 & 1.200566706 & 2.31961973   \\
$(9/32)n$
     & 0.5625000000 & 1.125000000 & 2.250000000   \\
relative error
     & 0.2968750000 & 0.06294253008 & 0.03001342380   \\
\hline
\end{tabular}
}
\end{center}
\vspace{1em}
\end{example}
\section{Asymptotics of Fourier-Laplace integrals}\label{sec:analysis}  

We conclude with the main analytic results behind the proofs of 
Theorems~\ref{asymptotics} and \ref{asymptotics_degenerate}.
Pemantle and Wilson's approach in \cite{PeWi2002} and our approach here to 
deriving asymptotics for the Maclaurin coefficients of $F$ requires an 
asymptotic expansion for integrals of the form
\[ 
\int_X u(t) \, e^{-\omega g(t)} dt
\]
as $\omega\to\infty$,
where $X\subseteq\reals^{d-1}$ is open and
$u$ and $g$ are complex valued functions on $X$.
Pemantle and Wilson used Watson's lemma and Morse's lemma to 
prove the existence of a full asymptotic expansion for these Fourier-Laplace 
integrals as they are commonly called but gave an explicit formula for the 
leading term only.
In contrast we use the two theorems below which give explicit formulas for 
all terms. 

All function spaces mentioned are complex valued.

The first theorem deals with stationary and nondegenerate points, that is, 
points $t_0$ such that $g'(t_0)=0$ and $\det (g''(t_0)) \neq 0$, respectively.

\begin{theorem}[{\cite[Theorem 7.7.5]{Horm1983}}]\label{Hormander}
Let $X\subset \reals^{d-1}$ be open, $u \in C_c^\infty(X)$, and 
$g\in C^\infty(X)$.
If $\Re g\ge 0$, $\Re g(t_0)=0$, $g$ has a unique stationary point 
$t_0 \in\supp u$, and $t_0$ in nondegenerate, 
then for every $N\in\pnats$ there exist $M>0$ such that 
\[
\int_X u(t) \me^{-\omega g(t)} dt
=
\me^{-\omega g(t_0)} 
\left( \det\left(\frac{\omega g''(t_0)}{2\pi} \right) \right)^{-1/2}
\sum_{k=0}^{N-1} 
\omega^{-k} L_k(u,g)
+O\left( \omega^{-(d-1)/2-N} \right)
\]
for $\omega>0$.
Here $L_k$ is the function defined in Theorem~\ref{asymptotics}
(but without the stipulation $t_0=0$).
Moreover, the big-oh constant is bounded when the partial derivatives of $g$ 
up to order $3(N +\lceil (d-1)/2 \rceil ) +1$ all stay bounded in supremum 
norm over $X$.
\end{theorem}

The second theorem deals with degenerate stationary points for $d=2$ and
is proved by adapting H\"{o}rmander's approach.

\begin{theorem}[{\cite[Theorem 1]{Elst2008}}]\label{Elst}
Let $X\subset \reals$ be open, $u \in C_c^\infty(X)$, and 
$g \in C^\infty(X)$.
If $\Re g\ge 0$, $\Re g(t_0)=0$, $g$ has a unique stationary point
$t_0 \in\supp u$, and $v\ge 2$ is least such that $\gtilde^{(v)}(t_0)\neq 0$,
then for every $N\in\pnats$ there exists $M>0$ such that
\[
\int_X u(t) \me^{-\omega g(t)} dt
=
\me^{-\omega g(t_0)} \, \frac{2 (a \omega)^{-1/v}}{v}
\sum_{k=0}^{N-1}
\omega^{-2k/v}  L^\text{even}_k(u,g) 
+
O\left( \omega^{-(2N+1)/v} \right)
\]
for $\omega>0$ and $v$ even, and
\[
\int_X u(t) \me^{-\omega g(t)} dt
=
\me^{-\omega g(t_0)} \frac{(|a| \omega)^{-1/v}}{v}
\sum_{k=0}^{N-1} 
\omega^{-k/v}  L^\text{odd}_k(u,g)
+
O\left( \omega^{-(N+1)/v} \right)
\]
for $\omega>0$ and $v$ odd.
Here $L^\text{even}_k$ and $L^\text{odd}_k$ are the functions defined in 
Theorem~\ref{asymptotics_degenerate} (but without the stipulation $t_0=0$). 
Moreover, the big-oh constants are bounded when the derivatives of $g$ 
up to order $(v+1)(N+1)+1$ all stay bounded in supremum norm over $X$.
\end{theorem}
\section*{Acknowledgments}
We thank A. F. M. ter Elst for kindly proving Theorem~\ref{Elst} on request 
and proofreading this article.
We also thank the anonymous referee for providing helpful recommendations.
\bibliographystyle{amsalpha}
\bibliography{combinatorics}
\end{document}